\documentclass[12pt]{amsart}
\usepackage{amsthm,amsmath,amssymb}
\usepackage{addfont}
\usepackage{mathrsfs}
\usepackage[T1]{fontenc}
\numberwithin{equation}{section}

\def\ca{{\mathcal A}}
\def\cb{{\mathcal B}}
\def\cc{{\mathcal C}}
\def\cd{{\mathcal D}}

\def\cf{{\mathcal F}}

\def\cas{{\mathcal S}}
\def\ct{{\mathcal T}}

\def\bc{{\mathbb C}}

\def\bn{{\mathbb N}}

\def\bz{{\mathbb Z}}

\def\d{\delta}
\def\e{\varepsilon}

\def\g{\gamma}

\def\l{\lambda}

\def\r{\rho}

\def\f{\varphi}

\theoremstyle{plain}
\newtheorem{lemma}{Lemma}[section]
\newtheorem{proposition}[lemma]{Proposition}
\newtheorem{theorem}[lemma]{Theorem}
\newtheorem{corollary}[lemma]{Corollary}
\theoremstyle{definition}

\newtheorem{remark}[lemma]{Remark}
\newtheorem{definition}[lemma]{Definition}

\begin{document}

\title[ Factorizations and complete boundedness ]{\textsc{ Bilinear forms, Schur multipliers,  complete boundedness  and duality }}

\author[E.~Christensen]{Erik Christensen}
\address{\hskip-\parindent
Erik Christensen, Mathematics Institute, University of Copenhagen, Copenhagen, Denmark.}
\email{echris@math.ku.dk}
\date{\today}

\begin{abstract}
Grothendieck's inequalities for operators and bilinear forms imply some factorization results for complex  $m \times n$  matrices. Based on the theory of operator spaces and completely bounded mappings we present norm optimal versions of these results and two norm optimal factorization results related to the Schur product.  We show that the spaces of respectively bilinear forms and Schur multipliers are conjugate duals to each other with respect to their completely bounded norms.  
 \end{abstract}

\maketitle

\section{Introduction and Notation}

A complex scalar valued $m \times n $ matrix $X$ may represent many different things in pure and applied mathematics. In this article we will focus on the interpretations of $X$ in  4 different ways \begin{itemize} 
\item[(i)] 
As the matrix for a linear mapping $F_X$ of the $n$  dimensional abelian  C*-algebra $\ca_n \,:= \,C(\{1, \dots, n\}, \bc)$ into the $m$ dimensional Hilbert space $\bc^m.$
\item[(ii)]
As the kernel for a bilinear form $B_X$ on the product $\ca_m \times \ca_n$ of C*-algebras given by $$B_X(a, b) \, := \, \sum_{i = 1}^m\sum_{j=1}^n X_{(i,j)}a(i)b(j) .$$
\item[(iii)] 
 As a  a linear mapping $S_X$ on $(M_{( m \times n)}(\bc), \|.\|_\infty)$ induced by Schur multiplication by $X$  - or entry wise multiplication - given by $$S_X(A)_{(i,j)}\, : = \, X_{(i,j)}A_{(i,j)}.$$ 
\item[(iv)] 
  As a  a bilinear mapping $T_X$ of $\ca_m \times M_{( m \times n)}(\bc)$ into the Hilbert space $\bc^n$ given as   $$T_X(a,B)_j \, : = \, \sum_{i =1}^m a(i) X_{(i,j)}B_{(i,j)}.$$ 
\end{itemize}
 The first three interpretations are very well known and have been studied in many ways for more than a century, but he fourth was imposed on us by the research done for the first ones in the preparation of this article. The research we mentioned is based on a closer look at the connections between the  classical Grothendieck inequalities  and the theory of completely bounded multilinear mappings.  In this way we have noticed that it is possible to make some of the existing results on factorization of matrices sharper. We will not define the concepts named {\em Grothendieck inequalities } or {\em complete boundedness } now, but leave that to the last part of the introduction. 
  The short version of the content of this article is as follows: Look at any of the interpretations  above of a scalar matrix and then use our recent  uniqueness result \cite{C2} for Stinespring representations of completely bounded multilinear mappings to obtain an optimal factorization of the matrix $X,$ which shows that the completely bounded norm of this particular operator is in fact a certain factorization norm applied to the matrix.   Then notice that in any of the cases (i) and (ii)  Grothendieck's inequalities and the theory of operator spaces imply that the completely bounded norm is dominated by Grothendieck's constant times the norm of the operator. In fact Grothendieck's constant  is the  minimal positive real that  may be used in all cases. 
 
    The well known  estimates for an  upper bound for the completely bounded norm come  in the cases (i) and (ii) from, respectively, Grot-hendiek's inequality for mappings from an abelian C*-algebra into a Hilbert space and  from his inequality on bilinear forms on a pair of abelian C*-algebras \cite{Gr}, \cite{Pi3}.  
  In the third  case Smith's work \cite{Sm} shows that the norm of a Schur multiplier equals it's completely bounded norm. In the fourth case our duality result implies  that  the completely bounded norm equals the norm.
    
  The norm optimality of the factorization results are based on the fact that the Stinespring representations for completely bounded linear or multilinear mappings contain a statement on optimality with respect to the completely bounded norm, \cite{Pa0}, \cite{CS}, \cite{PS}, \cite{Pa}. 
 This means - as we see it - that the norm optimal factorization results we obtain are all part of the theory of operator spaces and completely bounded mappings, whereas the existing factorization results   are consequences of  Grothendiek's fundamental work \cite{Gr}. 
   
We need a little more notation in order to make an explicit formulation of the  results we obtain. 
 We will use small Greek letters to denote vectors in a Hilbert space and write $\|\xi\|_2$ to denote the norm of the vector $\xi$ in some Hilbert space. For a vector $\xi $ in $\bc^n$ we let $\Delta_n(\xi)$ denote the diagonal matrix in $M_n(\bc),$ whose diagonal equals $\xi,$ and  we let the  expression $\xi_|$ denote the column matrix in $M_{(n,1)}(\bc)$ with entries $\xi_i,$ and likewise $\xi_{-}$ denotes the row matrix in $M_{(1,n)}(\bc)$ with entries from $\xi.$ In some instances the expressions $\xi_|$ and $\xi_{-}$ will denote the corresponding operators between the Hilbert spaces $\bc$ and $\bc^n.$ For a  matrix  $T$ of scalars we let $\|T\|_\infty $ denote it's operator norm, and  $\|T\|_2 $ denotes it's Hilbert-Schmidt norm defined as $\|T\|_2\, = \,  \big(\sum_{i,j} |T_{(i,j)}|^2\big)^\frac{1}{2}.$  We will use the important fact that for a vector $\xi $ in $\bc^n$ we have $\|\xi\|_2 \, = \, \|\xi_|\|_\infty \, = \, \|\xi_{-}\|_\infty.$  
We also need the concepts named column and row norms of a matrix. For a scalar $m \times n $ matrix $X$ it's column norm is  just the maximum over the norms of all the columns, and it is denoted $\|X\|_c.$  It follows by the definition of the product of matrices that $\|X\|_c^2 \, = \, \|\mathrm{diag}(X^*X)\|.$ The row norm  $\|X\|_r $ of $X$ is defined as the column norm of $X^*,$  and we have $\|X\|_r^2 \, = \, \|\mathrm{diag}(XX^*)\|.$

We return to the items (i), .. , (iv) and let $r$ denote the rank of $X,$ then the results of this article may  be presented as

 \begin{align}  \label{results}  
 \mathrm{(i)} \quad \,\, \, \,  & \quad \exists \xi \in \bc^n, \,  \exists T \in M_{(m,n)}(\bc) \\ 
 \notag X\,& = \, T\Delta_n(\xi), \\ \notag & \quad   \|T\|_\infty \|\xi\|_2 = \|F_X\|_{cb} \, \leq \, k_G^\bc\|F_X\|.  \\ \notag & \\ \notag
 \mathrm{(ii)} \quad \,\, \, \,  & \quad \exists \xi \in \bc^n, \, \exists \eta \in \bc^m ,\, \exists T \in M_{(m,n)}(\bc) \\  \notag  \, \, X  \,& = \, \Delta_m(\eta)^* T \Delta_n(\xi),\,  \\ & \notag \, \, \quad \|\eta\|_2\|T\|_\infty \|\xi\|_2 = \|B_X\|_{cb} \, \leq \, K_G^\bc\|B_X\|. \\ \notag & \\ \notag 
 \mathrm{(iii)} \quad \,\, \, \,  & \quad \exists L \in  M_{(r,m)}(\bc),  \, \exists R \in M_{(r,n)}(\bc) \\ \notag  X \,&  = \, L^*R,\, \text{rank}(L) \, = \,  \text{rank}(R) \, = \,\text{rank}(X)\,=\, r, \\ \notag  & \, \, \quad \|L\|_c\|R\|_c\, = \, \|S_X\|_{cb} = \|S_X\|. \\ \notag & \\ \notag 
 \mathrm{(iv)} \quad \,\, \, \,  & \quad \exists \g \in \bc^m \exists L \in  M_{(r,m)}(\bc),  \, \exists R \in M_{(r,n)}(\bc) \\ \notag  X \,&  = \, \Delta_m(\g)  L^*R,\, \text{rank}(L) \, = \,  \text{rank}(R) \, = \,\text{rank}(X)\,=\, r, \\ \notag  & \, \, \quad \|\g\|_2\|L\|_c\|R\|_c\, = \, \|T_X\|_{cb}.
\end{align}
 
A first look at the items (ii) and (iii)  does not show any relation between the 2 interpretations of a scalar matrix, but we will show that in some sense the two concepts may be considered as dual to each other with respect to the inner product on $M_{(m \times n)}$ given by
$$\forall X, Y \in M_{(m \times n)}(\bc): \, \langle X, Y \rangle \,:= \, \mathrm{Tr}_n (Y^*X).$$ 
A similar relations holds between the items (i) and (iv) and this was the reason why, we found  the interpretation $T_X$ of a matrix $X.$

   The present article focusses on the use  of the theory of operator spaces and completely bounded mappings, a subject which now is well described in the literature \cite{ER}, \cite{Pa}, \cite{Pi1}, \cite{Pi2}.  Pisier has made many deep and impressing contributions in this area of research and quite a few of them relate closely to some parts of this article, see \cite{Pi0},  \cite{Pi1}, \cite{Pi2}, \cite{PS} to mention a few.  
 The factorization aspect is discussed in the chapters 3 and 5 of \cite{Pi1} in a more abstract setting, but we have not tried to find the exact relations between those results and the ones we present here. 
 It is well known  that Grothendieck's original article \cite{Gr} contains results on operators which either factor through a Hilbert space or through  a commutative unital C*-algebra, but we have not tried to relate this general theme to our work. On the other hand, the basic results we present rely on the article by Grothendieck, which Pisier in the article \cite{Pi3} names the {\em résumé}. 
 In the {\em résumé } Grothendieck shows a factorization result, for bilinear forms on a product of two abelian C*-algebras, which now is known as the {\em  Grothendieck inequality.} 
 A reformulation of this inequality tells that there exits  a universal positive constant $K_G^\bc$ such that  any complex $m \times n $ matrix  with bilinear norm  $\|B_X\|$ may be factored as \begin{equation} \label{GI} X \, = \, \Delta_m(\eta)^* T \Delta_n(\xi),\,  \|\xi\|_2 = \|\eta\|_2 = 1, \, \|T\|_\infty \leq K_G^\bc\|B_X\|,
\end{equation} 
This factorization of $X$ is identical to   the one described in (\ref{results}) item (ii) except for the extension  $\|T\|_\infty \,= \, \|B_X\|_{cb} \, \leq \, K_G^\bc\|B_X\|,$ so in order to be more precise we will describe the concept named {\em completely bounded .} 

A bounded linear mapping $\f$ of a subspace  $\cas$ of operators on some Hilbert space $H$ into some $B(K)$ for some Hilbert space $K$  is said to be completely bounded if there exists a positive $c$ such that  for any natural number $k$ the mapping $\f_k := \f \otimes \mathrm{id}_{M_k(\bc)} : \cas\otimes M_k(\bc) \to  B(H)\otimes M_k(\bc) $ has norm at most $c.$ If $\f$ is completely bounded, it's completely bounded norm $\|\f\|$ is defined as the sup over the norms $\|\f_k\|.$ 
In \cite{CS} the notion of complete boundedness was extended to multilinear mappings between spaces of bounded operators on Hilbert spaces in the following way. For a bounded  bilinear mapping $\Phi : \cas_1 \times \cas_2 \to B(K)$ defined on the product of  a pair of operator spaces $S_1 \subseteq B(K_1)$ and $S_2 \subseteq B(K_2)$  we define $\Phi_k : \big(\cas_1 \otimes M_k(\bc)\big)  \times  \big(\cas_2 \otimes M_k(\bc) \big) \to B(K) \otimes M_k(\bc)$ by a formula, which is analogous to the matrix multiplication. 
\begin{align} \label{CbMult} &\forall A \in M_k(\cas_1) \, \forall B \in M_k(\cas_2) \, \forall i,j \in \{1, \dots , k\}: \\ \notag  &\Phi_k(
A, B)_{(i,j) }\, := \, \sum_{l =1 }^k \Phi(A_{(i,l)}, B_{(l,j)}) .\end{align}
In the proof of Theorem \ref{OpThm} we shall see that equation (\ref{GI}) implies that $\|B_X\|_{cb} \, \leq \, K_G^\bc\|B_X\|.$

Grothendieck's {\em résumé} \cite{Gr} also shows that that the Grothendieck inequality may be used to describe those complex $m \times n $ matrices which are contractions as Schur multipliers.  The following theorem is a consequence of Proposition 7  of \cite{Gr}.

\begin{theorem} \label{GrT} 
Any complex $m \times n$ matrix $X$ is contained in the closed convex hull of $m \times n$ scalar  matrices $Y$  of the form $y_{(i,j)} = \bar{ l_i } r_j$ with $ |l_i|\,  \leq \,( K_G^\bc)^\frac{1}{2}\|S_X\|^\frac{1}{2}$ and $ |r_j| \,\leq \,(K_G^\bc)^\frac{1}{2}\|S_X\|^\frac{1}{2}.$
\end{theorem} This theorem is presented as Theorem 3.2 in \cite{Pi3} and that article contains proofs, extensions and historical notes. 
This theorem was later improved to the following sharper result, lowering the constant to 1. There exists  vectors $\xi_j$ and $\eta_i$ in the unit ball of some Hilbert space such that $X_{(i,j) }\, = \|S_X\|  \, \langle \xi_j, \, \eta_i\rangle.$ 
Here we present an  improvement, of this factorization to one with $X=L^*R,$ such that both $L$ and $R$ are finite scalar matrices of the  same rank as $X,$ and the norms of all the columns in $L$ and $R$ are at most $\|S_X\|^{(1/2)}.$ 
This factorization comes easily from the existing literature on completely bounded mappings see  \cite{Pa} Theorem 8.7 (iii), but the statement on the ranks of $L$ and $R,$ seems not to be noticed before.

In section 2 we will give the details in the proofs of the results on optimal factorizations as mentioned in the abstract and discussed above.

In Section 3, we  show that the  compact convex sets in the  $m \times n $ complex matrices defined by $\cc \cb_{(m,n)}  := \{X \in M_{(m,n)}(\bc)\, : \, \|B_X\|_{cb } \leq 1 \} $ and $\cc\cas_{(m,n)} := \{ Y \in M_{(m,n)} (\bc) \, : \, \|S_Y\| \leq 1\}$ are polars of each other with respect to the inner product $\mathrm{Tr}(Y^*X).$ We will also show  that the polar of the set $\cc\cf_{(m,n)}  \, : = \, \{ X \in M_{(m,n)}(\bc) \, : \, \|F_X\|_{cb} \leq 1\,\}, $ equals the set $\cc\ct_{(m,n)} \, : = \, \{  Y \in M_{(m,n)}(\bc) \, : \, \|T_Y\|_{cb} \leq 1\,\}. $ 

\section{ Factorizations }  
In the first place wee look at a complex $m \times n$ matrix $X$ as the kernel for a linear mapping $F_X$  of $\ca_n$ to $\bc^m.$ 
The mapping $F_X$ does not map into an operator space right away,  but the operator space $M_{(m,1)} (\bc) $ is isometrically isomorphic  to $\bc^m
$ and hence equipped with a natural structure as an operator space. We will then in the rest of this article assume that $F_X$ is an operator defined on the C*-algebra $\ca_n $ with image in $M_{(m,1)}(\bc),$ which in turn is a subspace of the C*-algebra $M_m(\bc).$ We will first show that $F_X$ is completely bounded with respect to this operator space structure and then find 2 natural Stinespring representations of this mapping, such that the factorization result drops out. 

\begin{theorem}   \label{OpThm} 
For any $X$ in $M_{(m,n)}(\bc) $ the mapping $F_X: \ca_n \to M_{(m,1)}(\bc)$ satisfies $ \|F_X\|_{cb} \leq k_G^\bc \|F_X\|.$ \newline
There exist a unit vector $\xi$ in $\bc^n$ and a matrix $C$ in $M_{(m,n)}(\bc) $ such that $\|C\|_\infty \,= \, \|F_X\|_{cb} $ and $X = C \Delta_n(\xi).$ 

A  norm optimal Stinespring representation may be obtained as \newline  $F_X(A) \, = \, C\Delta_n(A)\xi_|.$

\end{theorem} 

\begin{proof}
It is well known that the classical {\em little Grothendieck inequality} implies that there exist a unit vector $\xi$ in $\bc^n$ and a matrix $C$ in $M_{(m,n)} (\bc) $ such that  $X \, = \, C \Delta_n(\xi),$ and  $\|C\|_\infty \leq k_G^\bc\|F_X\|.$ 
Then we may write $$ \forall a \in \ca_n : \quad F_X(a)  \, = \, C \Delta_n(a) \xi_|. $$   We recall that $\xi_|$ sometimes denotes a matrix in $M_{(n,1)}(\bc)$  and sometimes an  operator in $B( \bc, \bc^n).$  In this equation it denotes an operator and we  have obtained a Stinespring representation of $F_X$ which proves the first  statement of the theorem. 
We will leave the Stinespring representation we just found and create 2 other ones now. 
We know that there is a norm  optimal and also  minimal Stinespring representation of $F_X$ which we denote in the following way, 
\begin{equation} \label{1stS}
\forall a \in \ca_n: \quad F_X(a)\, =  \,  W^*\gamma(a) V,
\end{equation} such that $\g $ is a representation of $\ca_n$ on a Hilbert space $K,$  $V $ is in $B(\bc, K)$ and $W$ is in $B(\bc^m, K)$ and $\|W\|\|V\| \, = \, \|F_X\|_{cb}.$ This is the optimality, and the minimality means that $K \, = \, \mathrm{span} (\{\g(a)V1\, : \, a  \in  \ca_n\, \}).$ and $K \, = \, \mathrm{span}( \{ \g(a) W \eta\, : \, a \in \ca_n, \, \eta \in \bc^m\,\}).$ 
We define the vector $\Omega_n$ in $\bc^n$  as the vector where all entries equal 1. Then we may obtain a Stinespring representation of $F_X$ in the following way. 
\begin{equation} \label{2ndS}
\forall a \in \ca_n: \quad F_X(a) \, = \,  X \Delta_n(a) (\Omega_n)_|.
\end{equation}
This Stinespring representation is minimal unless span$(\{ \Delta_n(a)X^*\eta\, : \,a \in \ca_n, \, \eta \in \bc^m\,\})$ is not all of $\bc^n.$ So we have a minimal representation if and only if all columns in $X$ are non vanishing. It is clearly no lack of generality to assume that every column in $X$ is non trivial, and with this assumption  fulfilled, the Stinespring representation from (\ref{2ndS}) is minimal.  From \cite{C2} we know that the representations $\g$ and $\Delta_n$ of $\ca_n$ are unitarily equivalent, so we may as well assume that the Hilbert space $K$ from (\ref{1stS}) equals $\bc^n$ and that $\g \, = \, \Delta_n,$ so $$X\Delta_n(a) (\Omega_n)_| \,= \, W^* \Delta_n(a) V .$$ Define the vector $\xi$  in $\bc^n$ as $\xi \, := \, V1,$ then $\|\xi\|_2 \, = \| V\|$ and elementary algebra shows that $X \, = \, W ^*\Delta_n(\xi),$ so the theorem  follows  
\end{proof}

The previous theorem is one-sided and the reason is that a kernel for a linear mapping is usually written to the left of the argument. This calls for the following definition.

\begin{definition}
Let $X$ be a scalar $m \times n $ matrix then $G_X$ is defined as the linear mapping of $\ca_m$ to $M_{(1,n)}(\bc)$ given by $$\forall a \in \ca_m. \quad G_X(a) \, : = \, a_{-} X.$$
\end{definition} 

For a completely bounded mapping $T$ between self-adjoint spaces of operators you may define $T^\#$ as a completely bounded operator between the same spaces and with completely bounded norm $\|T^\#\|_{cb} \, = \, \|T\|_{cb}$ via the equation $T^\#(y)\, : = \, (T(y^*))^*.$  
It is quite easy to see that for an $m \times n$ matrix $X$ we will have $(F_{X^*})^\# \, = \, G_X.$  Based on this we apply Theorem \ref{OpThm} to $X^*$ and we get a factorization 
$X^* \, = \, D^*\Delta_n(\eta)^*$ such that
 $\|D^*\|_\infty\|\eta\|_2 = \|F_{X^*}\|_{cb} \, = \, \|G_X\|_{cb}, $ and we have proven the following theorem.

\begin{theorem}   \label{OpThm*} 
For any $X$ in $M_{(m,n)}(\bc) $ the mapping $G_X: \ca_m \to M_{(1,n)}(\bc)$ satisfies $ \|G_X\|_{cb} \leq k_G^\bc \|G_X\|.$ \newline
There exist a unit vector $\eta $ in $\bc^m$ and a matrix $D$ in $M_{(m,n)}(\bc) $ such that $\|D\|_\infty \,= \, \|G_X\|_{cb} $ and $X = \Delta_m(\eta)D.$ 

A  norm optimal Stinespring representation may be obtained as \newline  $G_X(A) \, = \, \eta_{-}\Delta_m(a) D.$

\end{theorem}

We will now focus on the complex $m \times n $ matrix as a kernel for a bilinear operator $B_X$  on $\ca_m \times \ca_n.$ The result we present and it's proof are similar to the ones we just presented.

\begin{theorem}   \label{BilThm}
For any $X$ in $M_{(m,n)}(\bc): \quad \|B_X\|_{cb} \leq K_G^\bc \|B_X\|.$ \newline
There exist a unit vector $\xi$ in $\bc^n,$ a unit vector $\eta$ in $\bc^m$ and a matrix  $C$ in $M_{(m,n)}(\bc) $ such that $\|C\|_\infty \,= \, \|B_X\|_{cb} $ and $X = \Delta_m(\eta)^* C \Delta_n(\xi).$ 

A norm optimal Stinespring representation may be obtained as \newline $B_X(A,B)\, =\, (\eta_|)^* \Delta_m(A) C \Delta_n(B )\xi_|.$
\end{theorem}  

\begin{proof}
It follows from the classical Grothendieck inequality that there exist unit vectors $\mu$ in $\bc^m,$  $\nu$ in $\bc^n$ and a matrix $D$ in $M_{(m,n)}(\bc) $ such that $\|D\|_\infty \, \leq \, K_G^\bc\|B_X\|$ and $X \, = \, \Delta_m(\mu)^*D\Delta_n(\nu).$ Elementary manipulations show that 
\begin{align} \label{St3}
\forall y \in \ca_m \, \forall z \in \ca_n  : \, B_X(y,z) \, & = \, \sum_{i = 1}^m \sum_{j = 1}^n (y_i\bar\mu_i) D_{(i,j)} (\nu_j z_j) \\ \notag
& = \, (\mu_|)^* \Delta_m(y) D\Delta_n (z) \nu_| .
\end{align} 
So we have obtained a Stinespring representation of $B_X$ which shows that $\|B_X\|_{cb} \, \leq \, \|\mu\|_2\|D\|_\infty\|\nu\|_2 \, \leq \, K_G^\bc\|B_X\|.$ We will take a norm optimal and minimal Stinespring representation of $B_X$ and use the following notation. There exist Hilbert spaces $K, L$ representations $\pi$ of $\ca_m$ on $K,$ $\r$ of $\ca_n$ on $L,$ and operators $R $ in $B(K ,\bc),$ $ S $ in $ B(L,K),$ $T$ in $ B(\bc, L)$ such that for any $ y$ in $\ca_m$ for any  $z $  in $\ca_n$ 
\begin{equation}  \label{OptBilSt}
B_X(y,z) \, = \, R\pi(y)S\r(z) T, \text{ and } \|R\|\|S\|\|T\| \, = \, \|B_X\|_{cb}. 
\end{equation}   
As above, we let $\Omega_m$ denote the vector in $\bc^m$   where all the entries are 1, and in analogy with  (\ref{2ndS}) we define a second Stinespring representation of $B_X$ by 
\begin{equation} \label{2ndBilSt} 
B_X(y,z) \, = \, ((\Omega_m)_|)^*\Delta_m(y) X \Delta_n(z) (\Omega_n)_|.
\end{equation}
This second Stinespring representation is minimal if the following 4 conditions are satisfied 
\begin{itemize}
\item[(i)] span$(\{\Delta_n(z)\Omega_n \, : \, z \in \ca_n\,\}) \, = \, \bc^n.$ 
\item[(ii)] span$(\{\Delta_m(y) X\Delta_n(z)\Omega_n\,:\, y \in \ca_m, \, z \in \ca_n\, \} \, = \, \bc^m.$ 
\item[(iii)] span$(\{ \Delta_m(y^*) \Omega_m \, : \, y \in \ca_m  \,\}) \, = \, \bc^m.$
\item[(iv)] span$(\{ \Delta_n(z^*) X^*\Delta_m(y^*) \Omega_m\, : \, y  \in \ca_m,\, z \in \ca_n\, \} \, = \, \bc^n.$
\end{itemize} 
Since all the entries in the vectors $\Omega_m$ and $\Omega_n$ are 1, it follows that the conditions (i) and (iii) are fulfilled. The condition (ii) is fulfilled if all the rows in $X$  are non vanishing,  and, analogously,  (iv) is satisfied if all the columns in $X$ are non trivial.  Since it will be no serious restriction to assume that all the columns and all the rows in $X$ are non vanishing, we will assume so, and (\ref{2ndBilSt}) gives a minimal Stinespring representation of $B_X.$ We may then apply item (ii) of Theorem 3.2 in \cite{C2} and then assume that in the Stinespring  representation (\ref{OptBilSt}) we have $K \,= \, \bc^m,$ $L \, = \, \bc^n,$ $\pi \, = \, \Delta_m,$ and $\r \, = \, \Delta_n.$ Then we remark that the commutant of $\Delta_m(\ca_m)$ equals $\Delta_m(\ca_m)$ and similarly for $\ca_n,$ 
 so when we apply item (v) of the same theorem we find that there exists a vector $\xi$ in $\bc^n$ such that $\Delta_n(\xi)(\Omega_n)_| = T.$ Then $\xi = T1,$ where $1$ is a unit vector in $\bc,$ so $\|\xi\|_2 = \|T\|. $
We also get that there exists a vector $\eta$ in $\bc^m$ such that $\Delta_m(\eta)(\Omega_m)_| \, = \, R^*.$ Then $\eta \, = \, R^*1$ and $\|\eta\|_2 \, = \, \|R\|.$ Some elementary algebra shows that $X \, = \Delta_m(\eta)^* S \Delta_n(\xi) $ and   from (\ref{OptBilSt}) it follows that $\|\eta\|_2\|S\|\|\xi\|_2 \, = \, \|B_X\|_{cb},$ and the theorem follows.
\end{proof}

We will now turn to the study of the norm of a Schur multiplier $S_X$ for a complex $m \times n $ matrix $X,$  and the factorization result we get in this connection. 
  The search for estimates of the Schur multiplier norm $\|S_X\|$ has a long history, and we do not intend to cover all the contributions. On the other hand several works give estimates based on  norms of the diagonals
  of some  positive matrices. The most famous estimate is of course the very first one by Schur \cite{Sc}, although he did not see it this way, but his result    tells that for a positive matrix $X$ the Schur multiplication  is a positive mapping and hence it's norm equals the norm of the diagonal. 
It seems to us that the  usage of the words {\em  row and column norms} in connection with norms of Schur multipliers  appears first in Davidson and Donsig's article \cite{DD}. The previous researchers expressed their estimates in terms of norms of diagonals of some positive matrices, but this is really  the same thing, as we  saw in the introduction.   In the article \cite{C1} we constructed a concrete Stinespring representation of the Schur product, which showed  that Schur multiplication is a completely bounded bilinear operator of completely bounded norm 1.  Here we will just reformulate a single result from \cite{C1}, which shows that the factorization of Theorem \ref{LR2} is optimal.
  
\begin{proposition}  \label{Opti} 
Let $l, m, n $ be natural numbers, $L$ and $R$ be  matrices in $M_{(l,m)}(\bc)$ and $M_{(l,n)}(\bc)$ then $\|S_{(L^*R)}\| \leq \|L\|_c\|R\|_c. $
\end{proposition}  
\begin{proof}Let $k \, = \, \max\{l, m, n\} $ and consider the matrices $L$ and $R$ to be matrices in $M_k(\bc),$ and let $A$ be a matrix in $M_k(\bc),$ then we will apply  Theorem 2.3 of \cite{C1}. That result deals with the Schur block product, but it applies of course to the ordinary Schur product as well. With the notation from \cite{C1} the two representations $\l$ and $\r$ of $M_k(\bc)$ on $\bc^k \otimes \bc^k$ do commute, so we have 
\begin{align*}
\|S_{(L^*R)}(A)\| \,& = \, \|V^*\l(L^*R)\r(A)V\| \\ 
& = \, \|V^*\l(L^*)\r(A)\l(R)V\| \\ 
& \leq  \, \|V^*\l(L^*)\|\|\r(A)\|\| \l(R)V\| \text{ by Lemma 2.6 of \cite{C1}}\\ 
& =  \, \|L\|_c\|\r(A)\|\|R\|_c, \\ 
\end{align*}
 and the proposition follows. 
\end{proof}
This result gives simple proofs to some of the previous results on norms of Schur multipliers.  C. Davis \cite{Da} studies the  multiplier norm $\|S_X\| $ when $X$ is a self-adjoint matrix and his upper  estimate on the multiplier norm is the norm of the diagonal of $|X|$ which is exactly $\| (|X|)^{\frac{1}{2}}\|_c^2.$ When $X$ is self-adjoint with polar decomposition $X \, =\, S|X|, $ then $X$ may be factored as $X = (|X|^\frac{1}{2})(S|X|^\frac{1}{2}),$ and Davis result follows as an application of Proposition \ref{Opti} to this factorization.  
 M. Walter gets in \cite{Wa} an upper bound for a general $X$ expressed in terms of  the norms of some diagonals. In our language Walter's result is $\|S_X\|_{cb} \leq \| (XX^*)^\frac{1}{2} \|_c \| (X^*X)^\frac{1}{2} \|_c.$ In our setting this corresponds to the factorization of an $X $ with polar decomposition $X = V|X|,$ as $X = \big((XX^*)^\frac{1}{2}\big)\big( V(X^*X)^\frac{1}{2}\big).$ Bo{\`z}ejko gives in \cite{Bo} a very short proof of Walter's result which is related to the one we use. The difference lies mostly in his use of the existing theory of Banach spaces.
  The article \cite{DD} by Davidson and Donsig contains a lot of information on problems related to our investigation, but we have not seriously tried to apply our methods to questions on which subsets  $S$ of  the Cartesian product $\bz \times \bz$  that will have the property that Schur multiplication with the characteristic function of $S$ will induce a completely bounded Schur multiplier.
From \cite{BCD} we know that when $J=\bn$ then the restriction of a bounded matrix to it's lower diagonal part is not a bounded mapping, but we have found  no way to show that this matrix can not be factored as $L^*R$ with $\|L\|_c < \infty$ and $\|R\|_c < \infty.$   We have got the idea that the remarkable results by Lust-Piquard \cite{LP} ought to be able to help us to understand some of the questions on Schur multipliers, which we have been interested in. Unfortunately we were not able do so. In \cite{Li} Livshits studies the Schur block product between block matrices with operator entries, and he obtains an inequality which in our setting is a consequence of the complete boundedness of the Schur block product.

The inclusion of $M_{(m, n)}(\bc)$ in $M_k(\bc)$ for a $k $ bigger than both $m $ and $n$ may be done in any fashion where we fix $m $ rows and $n$ columns in $M_k(\bc). $ Such an  inclusion is  obviously an  isometry with respect to the operator norm and there is  a conditional expectation of norm 1, with respect to the operator norms,  from  $M_k(\bc)$ onto the embedded copy of $M_{(m,n)}(\bc) $ which is given by multiplication by orthogonal projections from the left and from the right. Consequently, if we look at a complex $m \times n$ matrix $X,$  the Schur multiplier norm $\|S_X\| $ is the same on  both  $M_{(m,n)}(\bc)$ and  $M_k(\bc).$  Having this in mind we will, sometimes,  work in the setting of square matrices until we have obtained our result in this setting and then deduce the general result at the very end.

The basic result we use in the proof of the following theorem, in replacement of the inequalities by Grothendieck used in the proofs of the theorems \ref{OpThm} and \ref{BilThm}, is Smith's result \cite{Sm} that the completely bounded norm $\|S_X\|_{cb} $ equals the operator norm $\|S_X\|.$

 We mentioned above that Grothendieck's work \cite{Gr} yields a description of of the structure of a Schur multiplier of norm one, which except for a constant is analogous to the optimal on one which may be obtained via the use of the theory of completely bounded mappings as described in Pisier's book \cite{Pi1} and  Paulsen's book \cite{Pa}. All of this gives the following theorem except that the following theorem  has the form of a factorization result, which includes a statement on the ranks of the matrices involved. The statement on ranks is a consequence of the following lemma.

\begin{lemma} \label{ReducTl} 
Let $k, m,n,r  $ be natural numbers  $A$ an $m \times k$ complex  matrix,  $B$ a $k \times n $ complex  matrix, and $r$ the rank of the product $AB. $  Then there exist an $m \times r$  complex  matrix $L$ and an $r \times n$ complex  matrix $R$  such that $AB\, = \, LR, \,\|L\|_\infty\,\leq \, \|A\|_\infty,$ $\|R\|_\infty\, \leq \, \|B\|_\infty,$  $\|L\|_2\, \leq \, \|A\|_2, \, \|R\|_2 \, \leq \, \|B\|_2,   \|L\|_r\,\leq \, \|A\|_r$   and $\|R\|_c\,  \leq \, \|B\|_c$
\end{lemma} 
\begin{proof}
Let $E$ denote the range projection of $B$ and $F$ the support projection of $AE,$ then the rank of $AE $ is the same as the rank of $AB,$ so the rank of $F$ is $r$ and since $F \leq E$  
we have that the rank of $FB $ also is $r.$  Since $r \leq k$ there exists  an isometry $W$ in $M_{(k,r)}(\bc)$ such that $W^*W \, = \, I_{\bc^r},$  $WW^* \, = \, F,$ and then we can define $L\, : = \, AW$ and $R\, := \, W^*B$ with the desired properties. 
\end{proof}

\begin{theorem} \label{LR2} Let $X$ be in $M_{(m,n)}(\bc)$ with  rank  $r,$  then there exist matrices $L$  in $M_{(r\times m)}(\bc)$  and $R$ in $M_{(r,n)}(\bc) $ such that they both have rank $r,\, $   $$  L^*R\,  = \, X \text{ and } \|L \|_c \|R\|_c \, = \, \|S_X\|_{cb} = \|S_X\|.$$

Let $k = \max\{m,n\}$ and  consider $X, L $ and $R$ to be  in $M_k(\bc),$ with the same numbering.   The operator $S_X $ on $M_k(\bc)$  will have a norm  optimal Stinespring representation given  by  $$ S_X(A)\, = \,\big( V^*\l(L^*)\big)\r(A)\big(\l(R)V\big).$$ \end{theorem}

\begin{proof} It is tempting to present a proof of this which is based on the results from \cite{C2}, in the same way as we have done a couple of times just above. It is possible to follow this idea, but, unfortunately, the path we found this way is more complicated than the proof below which is an easy application of a well known result.
We return to the result \cite{Pi0} Theorem 5.1 (ii) or  \cite{Pa} Theorem 8.7 (iii) 
so there  exists a Hilbert space $K$ with some vectors $\xi_j, $ $1 \leq j \leq  n$ and $\eta_i, $ $1 \leq i \leq m $ in $K$  such that $\|\xi_j \|_2 \, \leq \, \|S_X\|^\frac{1}{2},  \, \|\eta_i\|_2 \, \leq \, \|S_X\|^\frac{1}{2}   $  and $X_{(i,j)} \, = \, \langle \xi_j, \eta_i\rangle.$   
Let $K_1$ denote  the  subspace of $K$ given as the linear span of the vectors $\{\xi_1, \dots , \xi_n\}\cup \{\eta_1, \dots , \eta_m\}. $  Let $ d$ be the dimension of $K_1 $ and let $\{ e_1, \dots , e_d\} $ be an orthonormal basis  $K_1 ,$ then 
we can build  a  scalar $d\times m$ matrix $L^0 = (L^0_{(t,i)}) $ and a scalar  $d \times n$ matrix $ R^0 = (R^0_{(t,j)})$  by the formulae \begin{align*}  L^0_{(t,i) } &:= \langle \eta_i, e_t\, \rangle \\
 R^0_{(t,j) } &:= \langle  \xi_j, e_t\, \rangle .  
\end{align*}    
 It follows from the  basic theory of Hilbert spaces and the equation $X_{(i,j)} \, = \, \langle \eta_j, \, \xi_i\rangle$   that $(L^0)^*R^0 = X$ and  each column in both $L^0$ and $R^0$ has norm at most $\|S_X\|^{\frac{1}{2}}.$  The Lemma \ref{ReducTl} then shows that there exist matrices $L $ in $M_{(r,m)}(\bc) $ and  $R$ in $M_{(r,n)}(\bc)$ both with rank $r$ such that $ X \, = \, L^*R$ and $\|L\|_c\|R\|_c \leq \|S_X\| = \|S_X\|_{cb}. $ We return to Theorem 2.3 of \cite{C1} and recall that the representations $\l$ and $\r$ commute, so the fact that $\l(X) \, = \, \l(L^*) \l(R)$ shows that we get a Stinespring representation for $S_X$ as claimed in the theorem. The Lemma 2.6 of the same article shows that the completely bounded norm $\|S_X\|_{cb}$ is at most $\|L\|_c\|R\|_c$ and the theorem follows.
 \end{proof}
 
We will now switch to the study of the bilinear operator $T_X.$  

\begin{theorem} \label{Tx} Let $X$ be in $M_{(m,n)}(\bc)$ with  rank  $r,$  then there exist a vector $\g$ in $\bc^m$ and matrices $L $ in $M_{(r,m,)}(\bc),$ $  R $  in $M_{(r,n)}k(\bc) $  such that they both have rank $r,\, $   $$ \Delta_m(\g) L^*R\,  = \, X \text{ and } \|\g\|_2\|L \|_c \|R\|_c \, = \, \|T_X\|\, = \, \|T_X\|_{cb}.$$

 For $k \, := \, \max\{m,n\},$   the operator $T_X $ on $\ca_k \times M_k(\bc)$  will have a norm  optimal Stinespring representation given  by  $$ T_X(a,B)\, = \,\big(\g_{-}\big) \Delta_k(a)\big( V^*\l(L^*)\big)\r(B)\big(\l(R)V\big).$$ \end{theorem}
  
  \begin{proof} We will prove the theorem in the case when $m\, = \, n\, = \, k$ and then discus the general case afterwards. 
  The equality $\|T_X\| \, = \, \|T_X\|_{cb} $ follows  as Corollary \ref{TcbN} to Theorem \ref{Duality}, and the proof below  is independent of this result.

In the first part of the proof we will assume that $X$ is invertible in $M_k(\bc)$ and prove the result in this case. A compactness argument will then give the general result for quadratic matrices.   Using the notation from \cite{C1} it is immediate to obtain a Stinespring representation of $T_X$ in the following way 
\begin{equation} \label{StTx1} 
\forall a \in \ca_k\, \forall B \in M_k(\bc): \quad T_X(a,B) \, = \, (\Omega_k)_{-}\Delta_k(a) V^* \lambda(X) \rho(B)V.
\end{equation}
We will show that this Stinespring representation is minimal, so we have to show that the following equations hold
\begin{itemize} 
\item[(i)] span $(\{\r(B)V\xi\, : \, B \in M_k(\bc), \, \xi \in \bc^k\,\}) \, = \, \bc^k \otimes \bc^k.$ 

\item[(ii)] span $(\{\Delta_k(a)V^* \l(X)\r(B)V\xi\, : \,a \in \ca_k,  B \in M_k(\bc), \, \xi \in \bc^k\,\}) \newline \quad \quad = \, \bc^k.$ 
\item[(iii)] span $(\{\Delta_k(a) \Omega_k \, : \, a \in \ca_k \}) \, = \, \bc^k.$ 

\item[(iv)] span$(\{ \r(B)\l(X^*)V\Delta_k(a)\Omega_k \,  :  \, a \in \ca_k, B \in M_k(\bc) \, \})  \newline = \, \bc^k \otimes \bc^k.$ 
\end{itemize}
To verify item (i)  we remark that $\r(e_{(i,j)} )V\d_j  \, = \,   (\d_j \otimes \d_i).$   
 Since $X$ is  invertible and $\l$ is unital, $\l(X)$ is invertible and the verification of item (ii) then follows from (i) and the fact that the range of $V^*$ is all of $\bc^k.$ It is obvious that item (iii) holds. Item (iv) follows from the invertibility of $\l(X^*),$ the fact that $\l$ and $\r$ commute and the items (i) and (iii).
  We may then choose a norm optimal and minimal Stinespring representation for $T_X,$ and, based on Theorem 3.2  of \cite{C2}, we may assume that this Stinespring representation also uses the the representations $\Delta_k$ of $\ca_k$  and $\r$ of $M_k(\bc).$ Hence there exist a unit vector $\g$ in $\bc^k$ an operator $Y$ in $B(\bc^k \otimes \bc^k , \bc^k) $ such that $\|Y\| \, = \, \|T_X\|_{cb}$ and an operator $Z$ in $B(\bc^k, \bc^k \otimes \bc^k)$ such that $\|Z\| \, = \, 1$ with the following property
  
  \begin{equation} \label{StTx2} 
 \forall a \in \ca_k\, \forall B \in M_k(\bc): \quad T_X(a,B) \, = \, (\g)_{-}\Delta_k(a) Y \rho(B)Z.
\end{equation} 
From Theorem 3.2 of \cite{C2} we know that there exists an operator $\hat Z$ in the commutant of $\r(M_k(\bc))$  such that $\hat Z V \,= \, Z.$ We know that the commutant of $\r(M_k(\bc))$   equals $\l(M_k(\bc))$ so there exists a matrix $R$ in $M_k(\bc)$ such that $\hat Z \,= \, \l(R)$ and we have $Z \, = \, \l(R)V.$ The equation $\|Z\|\, = \, 1$ then by Lemma 2.6 of \cite{C1} implies that $\|R\|_c \, = \, 1.$   Since $X$ is assumed to be invertible, all the rows of $X$ are non vanishing, and then we can see that all the entries in $\g$ have to be non trivial, and the commutativity of the algebra of diagonal matrices show that $\g_{ -}\Delta_k(a) \, = \, (\Omega_k)_{-}\Delta_k(a) \Delta_k(\g).$ Before we start the following string of computations we remind you that $V\d_j := \d_j \otimes \d_j ,$ so we have $V\Delta_k(\xi) \, = \, \l(\Delta_k(\xi))V$ for any vector $\xi .$ Then the results contained in the items (i) and (iii)  and the remarks from above may be inserted into the equations (\ref{StTx1}) and (\ref{StTx2}), such that we can get the following string of equations.  

\begin{align}
\notag \Delta_k(\g)Y \l(R)  \, & =\, V^*\l(X) \\ 
 \notag Y \l(R)  \, & =\, \Delta_k(\g)^{-1}V^*\l(X) \\  \notag Y \l(R)  \, & =\, V^*\l(\Delta_k(\g)^{-1} X) \\ \label{VY} 
(VY) \l(R)  \, & =\, VV^*\l(\Delta_k(\g)^{-1} X) 
\end{align} 
We will write the equation \ref{VY} in tensor form based on the matrix units $e_{(i,j)} \otimes e_{(m,n)}, $ so we will first write the  expression for each of the operators appearing in equation (\ref{VY}). \begin{align*}  VY \,& = \, \sum_{i,j,m,n} (VY)_{((i,j),(m,n))} (e_{(i,j) }\otimes   e_{(m,n)}) \\ 
 \l(R) \, &= \, \sum_{s,t}R_{(s,t)} (e_{(s,t)} \otimes I)\\ VV^*\l(\Delta_k(\g)^{-1}X) \,& = \, \sum_{u,v} \g_u^{-1} X_{(u,v)} (e_{(u,v)} \otimes e_{(u,u)}).
  \end{align*}
Then equation (\ref{VY}) becomes 
\begin{equation} \label{MatProd} 
\sum_{i,j, m, n, t}  (VY)_{((i,j), (m,n))} R_{(j,t)} (e_{(i,t)} \otimes e_{(m,n)}) \, = \, \sum_{(u,v)}  \g_u^{-1} X_{(u,v)} (e_{(u,v)} \otimes e_{(u,u)}). 
\end{equation}
From here we define a matrix $L$ in $M_k(\bc)$ by defining it's adjoint $L^*$ via the formula 
\begin{equation} \label{L*} 
(L^*)_{ (i,j)} \, := \, (VU)_{(i,j),(i,i))}.
\end{equation} 
We see that the row norm of $L^*$ is dominated by the row norm of $VY,$ which is dominated by the norm of  $Y,$ so $\|L\|_c \leq \|Y^*\|\,= \, \|T_X\|_{cb}.$ 
An elementary computation and a comparison with the equation (\ref{MatProd}) show that $L^* R \,= \, \Delta_k(\g)^{-1} X$ and $ \Delta_k(\g) L^*R = X$ with $\|\g\|_2 = 1,$ $\|L\|_c \leq \|T_X\|_{cb},$  $ \|R \|_c =1$ and $$ T_X(a,B) \, = \,\big( \g_{-}\big) \Delta_k(a) \big(V^*\lambda(L^*)\big) \r(B) \big(\l(R)V\big),$$
so the theorem is proven in the case when $X$ is invertible.   
 
If $X$ is not invertible, then there exists a partial iosmetry $W$ of the kernel of $X$ onto the cokernel which equals the kernel of $X^*,$ and then for any    $\e > 0$ the operator $X_\e \, ;= \, X + \e W$ is invertible. The equation (\ref{StTx1}) implies that $\|T_X - T_{X_\e}\|_{cb} \leq \e \sqrt{k}.$ Then,  when we apply the result of this theorem to $X_\e,$  we find a unit vector $\g_\e$ in  $\bc^k,$ and matrices  $L_\e,$ $R_\e$ in $M_k(\bc)$  such that $\|L_\e\|_c \leq \|T_X\|_{cb} + \e\sqrt{k}, \|R_\e\|_c \leq 1 $ and $X_\e \, = \, \Delta_k(\g_\e) L^*_\e R_\e .$ A compactness argument then shows that there exists a unit vector $\g$ in $\bc^k$ and matrices $L_0, \, R_0$ such that $\|L_0\|_c \leq \|T_X\|_{cb} ,$ $\|R_0\|_c \leq 1$ and $X \, = \, \Delta_k(\g) L^*_0 R_0.$ In order to obtain the rank condition, which was mentioned in the theorem, we apply Lemma \ref{ReducTl},  and the theorem follows, in the quadratic case, except from the statement $\|T_X\| \, = \, \|T_X\|_{cb},$ which waits for the Corollary  \ref{TcbN}. In the rectangular case when  $X $ is in $M_{(m,n)} (\bc),$ we embed $X$ as $\tilde X$  into $M_k(\bc)$ with the same numbering and zeros elsewhere. One can verify that $\|T_{\tilde X}\|_{cb} \,= \, \|T_X\|_{cb}$ so $\tilde X\, = \, \Delta_k(\g_0) (L_0)^* R_0$ for a vector $\g_0$ in $\bc^k,$ a matrix $L_0 $ in $M_{(r,k)}(\bc)$ and a matrix $R_0$ in $M_{(r,k)}(\bc)$ such that $\|\g_0\|_2\|L_0\|_c \|R_0\|_c \, = \|T_X\|_{cb}.$ Let $P_m$ and $P_n$  denote the orthogonal projections of $\bc^k$ onto the subspaces spanned by the first $m$ and first $n$ basis vectors in $\bc^k,$ then $X \, = \Delta_m(P_m\g_0) (L_0P_m))^*(R_0P_n)$ will give the claimed factorization. 
 \end{proof} 
\section{Duality} 

 With the notation from above we will look at complex valued $m \times n $ matrices and define 6 compact convex subsets of these matrices  by 
\begin{align} \label{UnBalls} 
\cb_{(m,n)} \, &:= \, \{X \in M_{(m,n)}(\bc) \, : \, \|B_X\| \leq 1 \} \\ \notag
\cas_{(m,n)}\, &:= \, \overline{\mathrm{conv}}\big( \{X \in M_{(m,n)}(\bc)\, : \, X_{(i,j)}   = \bar l_i r_j ,\, |l_i| =1,\,  |r_j| = 1\,\} \big)\\ \notag 
\cc\cb_{(m,n)} \, & := \,\{X \in M_{(m,n)}(\bc) \, : \, \|B_X\|_{cb}  \leq 1 \} \\ \notag
\cc\cas_{(m,n)} \, & := \, \{X \in M_{(m,n)}(\bc) \, :  \, \|S_X\|_{cb}\leq 1 \} 
\\ \notag 
\cc\cf_{(m,n)}\,& : = \,\{X \in M_{(m,n)}(\bc)\, : \, \|F_X\|_{cb} \, \leq \, 1\,\}
\\ \notag
\cc\ct_{(m,n)} \,& := \, \{X \in M_{(m,n)}(\bc)\, : \, \|T_X\|_{cb} \, \leq 1  \}.
\end{align}
We have, indirectly, met these   sets in Section 2, and it follows from the way they are defined,  that  all of them are  compact convex subsets of $M_{(m,n)}(\bc)$ which are invariant under multiplication by complex scalars in the unit disc. 

In this  investigation  we need a couple of observations, which we list as propositions. 

Theorem \ref{OpThm} implies the following result.
\begin{proposition} \label{Ffac} Let $m,n $ be natural numbers then 
\begin{align*}
& \cc\cf{(m,n)}  \\ & = \, \{ C\Delta_n(\xi)\, : \, C \in M_{(m,n)}(\bc), \, \xi \in \bc^n, \,s.\,  t.\,  \|C\|_\infty\|\xi\|_2 \leq 1 \}. \end{align*}
\end{proposition}

The Theorem \ref{BilThm} implies. 
 
 \begin{proposition}  \label{CBThm} Let $m $ and $n $ be natural numbers then \begin{align*}\cc\cb_{(m,n)} \, = \, \{  \Delta_m(\eta)^*C \Delta_n(\xi)\, :  \, C \in M&_{(m,n)}(\bc), \, \eta \in \bc^m, \, \xi \in \bc^n \, \\ & s.\, t. \, \|\eta\|_2\|C\|_\infty \|\xi\|_2 \leq 1\,\}.\end{align*} 
 \end{proposition} 
 
 \begin{proof}
 The only thing we are missing in the proof is that for an $m \times n$  matrix $X = \Delta_m(\eta)^*C \Delta_n(\xi ) $ we have  $\|B_X\|_{cb} \leq \|\eta\|_2\|C\|_\infty\|\xi\|_2,$ but that was shown in the lines following  equation (\ref{St3}). 
\end{proof}  

 A combination of Proposition \ref{Opti} and Theorem \ref{LR2} yields immediately the following result.
 
\begin{proposition} \label{ScFac} Let $m,n $ be natural numbers and $l  :=  \min\{m,\, n\}, $ then 
$$\cc\cas_{(m,n)}  \, = \, \{ L^*R\, : \, L \in M_{(l,m)}(\bc), \|L\|_c \leq 1,\, R \in M_{(l,n)}(\bc) ,  \|R\|_c \leq 1\}.  $$ 
\end{proposition}
Theorem \ref{Tx} implies the following proposition.  
\begin{proposition} \label{TxFac} Let $m,n $ be natural numbers and $l  :=  \min\{m,\, n\}, $ then 
\begin{align*} \cc\ct_{(m,n)}  \, = \, \{ \Delta_m(\g)L^*R\, &: \, \g \in \bc^m, \|\g\|_2 \leq 1, \, L \in M_{(l,m)}(\bc), \|L\|_c \leq 1,\\ & \quad  R  \in M_{(l,n)}(\bc) ,  \|R\|_c \leq 1\}.  
\end{align*}
\end{proposition}
  \begin{remark} The Lemma \ref{ReducTl} implies that in the propositions above the number $l$ can be removed, but the introduction of $l$ strengthens the lemma in the way that it shows that all operators in those unit balls may be obtained with this limitation imposed.  
\end{remark}   
  We will need the following observation in the proof of Theorem \ref{Duality}.  
  \begin{lemma} \label{LemT} 
Let $T$ be in $M_{(m,n)}(\bc)$ such that $\|T\|_2 \, = \, 1.$  
There exists a matrix $R$ with $\|R\|_c \,= \, 1 $ in $M_{(m,n)} (\bc)$ such that all columns have norm either 1 or 0 and a unit vector $\xi$ in $\bc^n$ with 
positive entries, such that $ T \, = \, R\Delta_n(\xi).$ 

If $R$ in $M_{(m,n)}(\bc) $ satisfies $\|R\|_c \, \leq \,1 $ and $\xi$ in $\bc^n$ has norm at most 1,
 then for $T\, : = \, R\Delta_n(\xi):$ 
$\|T\|_2 \, \leq \, 1.$ \end{lemma} 

\begin{proof}
Given a $T$ with $\|T\|_2 \, = \, 1,$ then define $\xi_j \, := (\sum_i |T_{(i,j)}|^2)^\frac{1}{2}$ and $$R_{(i,j)} \, :=\, \begin{cases} 0 \quad \quad \quad \text{ if }  T_{(i,j)} \, = \, 0,\\
T_{(i,j)}/\xi_j \,  \text{ if } T_{(i,j)} \neq 0.\end{cases}$$
It follows by a direct computation that  we have obtained $T = R\Delta_n(\xi)$ with the right properties. Similarly  for a $T$ of the form $R\Delta_n(\xi)$  we have  $\|T\|_2 \, \leq \,\|R\|_c\|\xi\|_2. $ 
\end{proof}

 We will now look at the inner product in $M_{(m,n)}(\bc)$ defined via the trace Tr$_n$ on $M_n(\bc)$ satisfying Tr$_n(I) = n$  by \begin{equation} \label{InP}
 \forall X, Y \in  M_{(m,n)}(\bc): \quad  \langle X, Y \rangle\, := \, \mathrm{Tr}_n(Y^*X).
  \end{equation}  As usual we define the polar $\cd^\circ$ of a subset $\cd$ of $M_{(m,n)}(\bc)$ via the expression
  \begin{equation} \label{Polar} 
  \cd^\circ \, := \, \{ X \in M_{(m,n)}(\bc) \, : \, \forall D \in \cd, |\langle X, D \rangle | \leq 1 \, \}.  
  \end{equation}

   The 6  sets defined in the definitions (\ref{UnBalls}) all have the property that they are equal to their bi-polars,  so the polar circle in the statements below may be moved to the other side.  
  
  \begin{theorem} \label{Duality}  for any pair of natural numbers $m, n$ we have the following relations 
  \begin{align} \label{PolB}
  \cb_{(m,n)} \, &= \, 
\big(\cas_{(m,n)}\big)^\circ \\ \label{PolCB} 
\cc\cb_{(m,n)} \, & = \, 
\big(\cc\cas_{(m,n)}\big)^\circ \\
 \label{PolCF} 
\cc\cf_{(m,n)} \, & = \, 
\big(\cc\ct_{(m,n)}\big)^\circ.
\end{align} 
\end{theorem} 

\begin{proof} The equation (\ref{PolB}) is based on the identity $$\mathrm{Tr}_n\big( (\bar l_i r_j)^*X\big) \, = \, \sum_{i, j}  l_i x_{(i,j)} \bar r_j$$ and then it follows from the definition of the norm of $B_X$ as a bilinear operator on the pair of C*-algebras $\ca_m \times \ca_n,$ and the fact that the extreme points in the unit-ball of these algebras are the unitaries.
 With respect to (\ref{PolCB}) we will introduce the conjugation operation on $M_{(m,n)}(\bc)$ which is defined by $(\overline{X})_{(i,j) }\, := \, \overline{X_{(i,j)} }.$ Since the transposition and the adjoint operation both are isometries with respect to the operator norm, it follows that the conjugation operation is a conjugate linear isometry on $M_{(m,n)} (\bc)$ equipped with the operator norm and hence we get that $\|S_X\| = \|S_{\overline{X}}\|.$ An elementary calculation shows that for matrices $C, X$ in $M_{(m,n)}(\bc)$ and vectors $\eta$ in $\bc^m,$ $\xi $ in $\bc^n$ we have 
 \begin{equation} \label{dual} 
 \mathrm{Tr}_n\big(X^*\Delta_m(\eta)^*C \Delta_n(\xi)\big) \, = \, \langle \big(\bar X \circ C \big) \xi, \, \eta \rangle.
 \end{equation}
The equality in Proposition  \ref{CBThm} may be applied to the equation just above, and  when (\ref{dual}) is  read  from the left to the right we find  that $(\cc\cb_{(m,n)})^\circ \subseteq \cc\cas_{(m,n)}.$ On the other hand, when read from the right to the left it shows that $\cc\cas_{(m,n)} \subseteq (\cc\cb_{(m,n)})^\circ.$ This concludes the proof of the equations (\ref{PolB}) and (\ref{PolCB}), but  the proof of equation (\ref{PolCF}) is a bit more complicated. 

 We have noticed that both of the sets $\cc\cf_{(m,n)}$ and $\cc\ct_{(m,n)}$  equal their bipolar, so it is sufficient to prove that $$\cc\cf_{(m,n)}\, \subseteq \cc\ct_{(m,n)}^\circ  \text{ and }  \big(M_{(m,n)} (\bc)  \setminus \cc\cf_{(m,n)} \big)  \subseteq  \big( M_{(m,n)}(\bc) \setminus \cc\ct_{(m,n)}^\circ \big).$$
 To show the first inclusion we choose $X$  in $\cc\cf_{(m,n)}$ and $Y$ in $\cc\ct_{(m,n)}.$ Then $X = C\Delta_n(\xi) $ with $\|C\|_\infty \leq 1, $  $\|\xi\|_2 \leq 1, $  and $Y = \Delta_m(\eta) L^*R$ with $\|\eta\|_2 \leq 1$ and $  \|L\|_c \|R\|_c \leq 1.$  Then  by equation (\ref{dual}) we have  \begin{align*}  |\langle Y, X \rangle | \,& = \, |\mathrm{ Tr}\big(\Delta_n(\xi)^* C^* \Delta_m(\eta) (L^*R)\big)| \\ & = | \langle (\bar C \circ (L^*R)) \bar \xi, \, \bar \eta \rangle | \, \leq \,1, \end{align*}    
 so $\cc\cf_{(m,n)} \subseteq \cc\ct_{(m,n)}^\circ.$  
 
 Let us now suppose that $X$ in $M_{(m,n)}(\bc)$ is not in $\cc\cf_{(m,n)}$ then $\|F_X\|_{cb} > 1$ and recall that we look at $F_X $ as an operator with image in the operator space $M_m(\bc).$ Hence by Smith's result  \cite{Sm}, \cite{Pa} Proposition 8.11 we have $\|F_X\|_{cb} = \|(F_X)_m\|.$  
 This means that there exists a contraction $A = (A_{(i,j)})$ in $M_m(\ca_n),$ a unit vector $\Lambda $ in $\bc^m$ and a vector $\Theta$ in the sum of $m$ copies of $\bc^m$ with  $\Theta = (\Theta_1, \dots , \Theta_m) $ $\Theta_s = ( \theta_{(s,1)} , \dots , \theta_{(s,m)})$ and $\sum_{s,i}  |\theta_{(s,i)}|^2 = 1$ such that \begin{equation} \label{Domination} 
 \big| \langle (F_X)_m (A) \Lambda , \, \Theta \rangle \big| > 1.
\end{equation}  
We can define a column operator $B$ in $M_{(m,1)}(\ca_n ) $ by $B\, := \, A \Lambda$ and then $B$ is a contraction which means that we may define an $ m \times n$ scalar matrix  $R$  which is column norm bounded by $1$ as  \begin{equation}
\label{Ld} R \in M_{(m,n)} (\bc), \quad R_{(s,j)} \, = \, \overline{B_{(s,1)}(j) }.
\end{equation} 
The unit vector $\Theta$ may be used to define an $m \times m$ matrix $T$ such that $\|T\|_2 = 1 $ by 
\begin{equation}
\label{Td} T \in M_m(\bc) , \quad T_{(i,s)} \, :=\,  \theta_{(s,i)}. 
\end{equation} We may then compute 
\begin{align} \label{Dual3} 
1 \, & < \, \big| \langle (F_X)_m (A) \Lambda , \, \Theta \rangle \big|  \\  \notag &= \, 
\big| \sum_{s,t =1}^m ( \sum_{i=1}^m \sum_{j=1}^n X_{(i,j)}A_{(s,t)}(j) \l_t \bar \theta_{(s,i)} )\big|  \\  \notag &= \, 
 \big| \sum_{s =1}^m ( \sum_{i=1}^m \sum_{j=1}^n X_{(i,j)}B_{(s,1)}(j) \bar \theta_{(s,i)}) \big|  \\  \notag & =\,
  \big| \sum_{i=1}^m \sum_{j=1}^n X_{(i,j)} \sum_{s =1}^m (R^*)_{(j,s)}(T^*)_{(s,i)}  \big|  \\  \notag & =\,
  \big|  \sum_{i=1}^m \sum_{j=1}^n X_{(i,j)}(R^*T^*)_{(j,i)}  \big|  \\  \notag & =\,
  \big|\langle X, TR \rangle \big|  .
\end{align}
Recall that $\|T\|_2 = 1, $ in $M_m(\bc)$ 
and then by Lemma \ref{LemT} there exists a unit vector $\eta $ in $\bc^m$ and an operator $L$ in $M_m(\bc)$ such that $\|L\|_c = 1 $ and $T^* \, = \, L\Delta_m(\eta)^*,$ and then $T \, = \,\Delta_m(\eta) L^*.$ If this equation is inserted into the equation (\ref{Dual3}) we see that  
$X$ is not in the polar $\cc\ct_{(m,n)}^\circ$ and the theorem follows.
\end{proof}

The equation (\ref{PolCF}) implies the following corollary, which completes the proof of Theorem \ref{Tx}.

\begin{corollary} \label{TcbN}
Let $X$ be a complex $m \times n$ matrix, then $\|T_X\| = \|T_X\|_{cb}.$ 
\end{corollary} 

\begin{proof}
Let $C$ be an $m \times n$ complex matrix of operator norm at most 1 and $\xi$ a vector in the unit ball of $\bc^n,$  then $\overline{ C}$ is an operator of norm at most 1 and we have \begin{align*}
 \|T_X\| \, & \geq \,\|T_X(I_{\ca_m}, \overline{C}) \| \\
  &  \geq \,  |\langle T_X(I_{\ca_m}, \bar C), \xi \rangle | 
  \\  & = \,\big|\,  \sum_j \bar\xi_j \sum_i X_{(i,j)}\overline{C_{(i,j)}}\, \big|\\
  &= \, \big| \mathrm{Tr}_n((C\Delta_n(\xi))^*X) \big|\\ &= \, \big| \langle X, \, C\Delta_n(\xi) \rangle \,\big|.
\end{align*} 
From Proposition \ref{Ffac} and equation (\ref{PolCF})  we then see that $\|T_X\| \geq \|T_X\|_{cb},$ and the corollary follows.

\end{proof}
The Theorem \ref{Duality} has an immediate corollary  for bilinear forms.  That result may be deduced from the existing literature, but since it is a direct consequence of the previous proof we find it reasonable to have it included. 

\begin{corollary} \label{COR}
Let $X$ be in $M_{(m,n)}(\bc)$ and $l\, := \, \min\{ m,  \, n\},$  then $\|B_X\|_{cb} \, = \, \|(B_X)_l\|.$  
\end{corollary} 

\begin{proof}
By Theorem \ref{Duality} and compactness  there exists $Y$ in \newline  $\cc\cas_{(m,n)}( \bc)$  such that $ \|B_X\|_{cb} \, = \, \mathrm{Tr}_n(Y^*X).$ Then by Proposition \ref{ScFac} there exists $L $ in $M_{(l,m)}(\bc)$ and $R$ in $M_{(l,n)}(\bc)$ both with column norms at most $1,$ such that $Y \, = \, L^*R.$  
We can then construct contraction matrices $ A $ in $M_l ( \ca_m)$ and $B$ in $M_l(\ca_n),$  such that 
$\|(B_X)_l (A, B) \| \, = \, \|B_X\|_{cb} $ in the following way. 

For $ 1 \leq s \leq l$ define elements $a^s$ in $\ca_m$ and $b^s$ in $\ca_n $ by $$ a^s(i) \, : = \, l_{(s,i)} \text{ and } b^s(j), \,:= \, \overline{r_{(s,j)}} .$$ Since the column norms of $L$ and $R$ are at most 1 we have 
\begin{equation} \label{CRI} \sum_{s =1}^l a^s(a^s)^* \, \leq \, I_{\ca_m} \text{ and }  
\sum_{s=1}^l  (b^s)^*b^s \, \leq \, I_{\ca_n} ,\end{equation}  and we can define matrices  $A$ in $M_l(\ca_m)$ and $B$ in $M_l(\ca_n).$
\begin{align} \label{AB} 
A_{(u,v)}\, &:= \, \begin{cases} a^v \, \text{ if } u\, = \, 1,\\
0 \, \, \text{ if } u \, \neq \, 1
\end{cases}\\ \notag
B_{(v,w)}\, &:= \, \begin{cases} b^v\, \text{ if } w\, = \, 1,\\
0 \, \, \text{ if } w \, \neq \, 1.
\end{cases}
\end{align}  
Since $A$ is a one row operator and $B$ a one column operator the inequalities (\ref{CRI}) imply that $A$ and $B$ are contractions and then  \newline $\|(B_X)_l(A,B) \|\,  \leq \, \|(B_X)_l\|$ On the other hand since $A$ is  a one row matrix and $B$ is a one column matrix we have $$\|(B_X)_l(A,B)\|  \, = \,|\big(B_X)_l(A,B)\big)_{(1,1)}| \, = \,  |\sum_{s=1}^l B_X(a^s, b^s)|.$$ Then we have the following equations 
\begin{align*}
\|(B_X)_l\| \, & \geq \, \big| \sum_{s=1}^l B_X(a^s,b^s) \big| \\ \notag & = \, 
\big| \sum_{s=1}^l \sum_{i =1}^m \sum_{j=1}^n X_{(i,j)} a^s(i)b^s(j) \big| \\ \notag & = \, 
 \big| \sum_{s=1}^l \sum_{i =1}^m \sum_{j=1}^n X_{(i,j)}l_{(s,i)}\overline{r_{(s,j)}} \big| \\ \notag & = \, \big|\mathrm{Tr}_m(X(R^*L))\big|
\\ \notag & = \, \big|\mathrm{Tr}_n(Y^*X)\big| 
 \\ \notag & = \, \|B_X\|_{cb}  
\end{align*} and the corollary follows.
\end{proof}

 \begin{remark} \label{RemRc} 
The proof of the corollary \ref{COR} actually shows that the maximum is attained in the case where $A$ is a one row matrix of length $l$ and $B$ is a one column matrix of length $l,$ so the completely bounded norm is given as \begin{equation}
\|B_X\|_{cb} \, = \, \max\{ \,|\sum_{s = 1}^l B_X(a_s,b_s) |\, : \, \sum_{s = 1 }^l a_sa_s^* \leq I_{\ca_m} \, \, \sum_{s =1 }^l b_s^*b_s \leq I_{\ca_n}\, \}.
\end{equation}
The finite sums above may have some implications for the Haagerup tensor product $\ca_m \underset{h}\otimes \ca_n,$ or the content of the remark may follow from well known properties of the Haagerup tensor product   and the finite dimensionality of the factors in the tensor product, see \cite{Pa} Ch. 17 or \cite{Pi3} Theorem 14.1. 
The article \cite{PS} studies bilinear forms on $\ca_n \times \ca_m $ from the point of view that such a bilinear form may be considered as an operator from the operator space $\ca_n$ to the operator space given as the dual space of  $\ca_m.$  We have not pursued this aspect here since the algebras $\ca_n$ and $\ca_m$ are abelian, and in this case the two versions of complete boundedness for bilinear forms do agree.   
\end{remark}

\bigskip
Our last comment falls out naturally from some of the results above, and it may be  known to some researchers, but we have not seen it formulated  this way elsewhere.  
 We recall that Grothendieck's inequality means that $\cb_{(m,n)}\subseteq  K_G^\bc \cc\cb_{(m,n)},$ and $K_G^\bc$ is the smallest positive real which works for all pairs $(m,n).$ The result on polars then implies the well known result that $\cc \cas_{(m,n)} \subseteq  K_G^\bc \cas_{(m,n)}.$ The  results in the Theorems \ref{OpThm},  \ref{BilThm}, Smith's result \cite{Sm} and Corollary \ref{COR}  imply that Grothendieck's complex constants $k_G^\bc $ and $K_G^\bc$  formally may be computed  as 
\begin{theorem} \label{kgKg} 
\begin{align} k_G^\bc\, &= \, \underset{k\in \bn}{  \sup}\, \, \underset{X \in M_k(\bc)}{\sup} \frac{\|(F_X)_k\|}{\|F_X\|}, \\
K_G^\bc \, &= \, \underset{k\in \bn}{  \sup}\, \,\underset{X \in M_k(\bc)}{\sup} \frac{\|(B_X)_k\|}{\|B_X\|}.
\end{align}
\end{theorem}

\section*{Acknowledgement}
 We are happy to thank  Narutaka Ozawa, Vern Paulsen and Gilles Pisier for very valuable help and comments.

\end{document}